\documentclass[11pt,a4paper]{article}

\usepackage{epsf,epsfig,amsfonts,amsgen,amsmath,amstext,amsbsy,amsopn,amsthm
}
\usepackage{ebezier,eepic}
\usepackage{color}
\usepackage{multirow}
\newtheorem{thm}{Theorem}[section]
\newtheorem{theoremA}{Theorem}

\newtheorem{conjectureA}{Conjecture}

\newtheorem{lem}{Lemma}[section]
\newtheorem{false statement}{False statement}
\newtheorem{cor}{Corollary}

\theoremstyle{definition}

\newtheorem{claim}{Claim}

\newtheorem{case}{Case}
\newtheorem{subcase}{Case}[case]
\newtheorem{subsubcase}{Case}[subcase]

\def\qed{\hfill \rule{4pt}{7pt}}

\newcounter{mathitem}
  {\begin{list}{{$(\roman{mathitem})$}}{
   \setcounter{mathitem}{0}
   \usecounter{mathitem}
   \setlength{\topsep}{0pt plus 2pt minus 0pt}
   \setlength{\parskip}{0pt plus 2pt minus 0pt}
   \setlength{\partopsep}{0pt plus 2pt minus 0pt}
   \setlength{\parsep}{0pt plus 2pt minus 0pt}
   \setlength{\leftmargin}{35pt}
   \setlength{\itemsep}{0pt plus 2pt minus 0pt}}}
  {\end{list}}

\begin{document}

\title{\bf Exact bipartite Tur\'an numbers of large even cycles}

\date{}

\author{Binlong Li\thanks{Department of Applied Mathematics,
Northwestern Polytechnical University, Xi'an, Shaanxi 710072,
P.R.~China. Email: binlongli@nwpu.edu.cn. Partially supported
by the NSFC grant (No.\ 11601429).}~~~ Bo
Ning\thanks{Corresponding author. Center for Applied
Mathematics, Tianjin University, Tianjin 300072, P.R.
China. Email: bo.ning@tju.edu.cn. Partially supported
by the NSFC grant (No.\ 11601379) and
the Seed Foundation of Tianjin University (2018XRG-0025).}}
\maketitle

\begin{center}
\begin{minipage}{140mm}
\small\noindent{\bf Abstract:} Let the bipartite Tur\'an number
$ex(m,n,H)$ of a graph $H$ be the maximum number of edges in an
$H$-free bipartite graph with two parts of sizes $m$ and $n$,
respectively. In this paper, we prove that $ex(m,n,C_{2t})=(t-1)n+m-t+1$ 
for any positive integers $m,n,t$ with $n\geq m\geq t\geq \frac{m}{2}+1$. 
This confirms the rest of a conjecture of Gy\"{o}ri \cite{G97} (in a
stronger form), and improves the upper bound of $ex(m,n,C_{2t})$
obtained by Jiang and Ma \cite{JM18} for this range. We also prove a
tight edge condition for consecutive even cycles in bipartite
graphs, which settles a conjecture in \cite{A09}. As a main tool,
for a longest cycle $C$ in a bipartite graph, we obtain an estimate
on the upper bound of the number of edges which are incident to at
most one vertex in $C$. Our two results generalize or sharpen a
classical theorem due to Jackson \cite{J85} in different ways.

\end{minipage}
\end{center}

\section{Introduction}
We only consider simple graphs, which are undirected and finite.
The study of cycles in bipartite graphs has a rich history.
There are many results in the literature which use that a bipartite
graph with high degree has a long cycle, see references \cite{MM63,BC76,J81,J85,JL94}.
Moreover, \cite{KV05,LK18} reveal that results on cycles in bipartite
graphs play important roles in investigating
cycles in hypergraphs. For a lot of references from the view of
extremal graph theory, we refer to the survey \cite{FS13}.

Maybe one of the best known extremal results involving long cycles
in bipartite graphs is the following proved more than
30 years old.
\begin{thm}[Jackson \cite{J85}]\label{Thm-Jackson}
Let $t$ be an integer and $G=(X,Y;E)$ be a bipartite graph. Suppose
that $|X|=n$, $|Y|=m$, where $n\geq m\geq t\geq 2$.
Suppose that
$$e(G)>\left\{\begin{array}{ll}
  (n-1)(t-1)+m,    & if~m\leq 2t-2;\\
  (m+n-2t+3)(t-1),  & if~m\geq 2t-2.
\end{array}\right.$$
Then $G$ contains a cycle of length at least $2t$.
\end{thm}

One question naturally arises: Can we find
exact edge number conditions for cycles of given lengths? As we
shall see later, we indeed have the following significant
strengthening of Jackson's theorem.

\begin{thm}\label{Thm-SharpenJackson}
Let $t$ be an integer and $G=(X,Y;E)$ be a bipartite graph
with $|X|=m$, $|Y|=n$. Suppose that $n\geq m$ and
$t\leq m\leq 2t-2$. If $e(G)>(t-1)(n-1)+m$, then $G$ contains a
cycle of length $2t$.
\end{thm}

The above theorem in fact tells us exact information on
``bipartite Tu\'an number" of large even cycles.
Following F\"{u}redi \cite{F96}, we define the bipartite
Tur\'an number $ex(m,n,H)$ of a graph
$H$ to be the maximum number of edges in an $H$-free bipartite graph
with two parts of sizes $m$ and $n$.
In this paper, we mainly focus on the exact formula of $ex(m,n,C_{2t})$ for some range.
For a similar problem on paths, Gy\'{a}rfas, Rousseau, and
Schelp \cite{GRS84} completely determined the function $ex(m,n,P_t)$.
When restricting to cycles, the situation turns out to be much more difficult.

Let us recall the classical result that
$ex(n,n,C_4)=(1+o(1))n^{\frac{3}{2}}$ due to K\H{o}v\'{a}ri,
S\'{o}s, and Tur\'an \cite{KST54}. For the function $ex(m,n,C_6)$,
it is closely related to a number-theoretical problem on product
representations of squares, which was studied by Erd\H{o}s,
S\'{a}rk\"{o}zy, and S\'{o}s in \cite{ESS95}. They conjectured that:
(a) $ex(m,n,C_6)<c(mn)^{\frac{2}{3}}$ when $n> m\geq
n^{\frac{1}{2}}$; and (b) $ex(m,n,C_6)<2n+c(mn)^{\frac{2}{3}}$ when
$n\geq m^2$. The part (a) of this conjecture and a weaker result of
part (b) were confirmed by S\'{a}rk\"{o}zy \cite{S95}. For the part
(b), it was finally settled by Gy\"{o}ri \cite{G97}. Interestingly,
motivated by the extremal result on short cycles, Gy\"{o}ri
\cite{G97} suggested a general conjecture on longer cycles.

\begin{conjectureA}[{\rm Gy\"{o}ri \cite[p.373]{G97}, see also \cite{BGMV07}}]\label{Conj-Gyori}
Suppose that $m,n,k$ are integers, where $n\geq m^2$,
$m\geq t\geq 3$. Then
$$
ex(m,n,C_{2t})\leq (t-1)n+m-t+1.
$$
\end{conjectureA}
Using estimate on total weights of triangle-free multi-hypergraphs,
Gy\"{o}ri himself \cite{G06} disproved Conjecture \ref{Conj-Gyori}
for the case $t=3$. Balbuena, Garc\'{i}a-V\'{a}zquez, Marcote,
and Valenzuela \cite{BGMV07} further disproved it when $t\leq \frac{m+1}{2}$.
As far as we know, this conjecture remains open when $t\geq \frac{m}{2}+1$.
The following conjecture sharpens the left part of Gy\"{o}ri's conjecture.

\begin{conjectureA}\label{Conj-Cons}
Suppose that $m,n,t$ are positive integers, where $n\geq m\geq t\geq \frac{m}{2}+1$. Then
$$ex(m,n,C_{2t})=(t-1)n+m-t+1.$$
\end{conjectureA}

For general results on $ex(m,n,C_{2t})$, an upper bound was obtained by Naor and Verstra\"{e}te \cite{NV05},
who proved that for $m\leq n$ and $t\geq 2$,
$$ex(m,n,C_{2t})\leq\left\{\begin{array}{ll}
  (2t-3)\cdot [(mn)^{\frac{t+1}{2t}}+m+n],    & \mbox{if }t\mbox{ is odd};\\
  (2t-3)\cdot [m^{\frac{t+2}{2t}}n^{\frac{1}{2}}+m+n],  & \mbox{if }t\mbox{ is even}.
\end{array}\right.$$
Gy\"{o}ri \cite{G97} proved that there exists some
$c_t>0$ such that for $n\geq m^2$,
$$
ex(m,n,C_{2t})\leq (t-1)n+c_t\cdot m^2.
$$
Very recently, Jiang and Ma \cite{JM18} proved the following new bound:
\begin{thm}[Jiang and Ma, Proposition 5.5 in \cite{JM18}]
There exists a constant $d_t>0$ such that for any positive integers $n\geq m\geq 2$,
$$
ex(m,n,C_{2t})\leq (t-1)n+d_t\cdot m^{1+\frac{1}{[t/2]}}.
$$
\end{thm}
So, if Conjecture \ref{Conj-Cons} is true, then it improves
Jiang and Ma's result for the range $t\geq \frac{m}{2}+1$.

In this paper, we aim to solve the aforementioned conjectures.
In fact, we prove the following stronger result, which also confirms a
conjecture in \cite{A09} (see Conjecture 1 in \cite[p.30]{A09}).
\begin{thm}\label{Thm-general}
Let $G=(X,Y;E)$ be a bipartite graph with $|X|= m$ and $|Y|=n$.
Suppose that $n\geq m\geq 2k+2$ for some $k\in\mathbb{N}$. If
$e(G)\geq n(m-k-1)+k+2$, then $G$ contains cycles of all even
lengths from 4 up to $2m-2k$.
\end{thm}

Set $t=m-k$ in Conjecture 2. Let $G$ be a graph obtained by
identifying one vertex in the $n$-set from $K_{m-k-1,n}$ and the
other vertex in the $1$-set from $K_{1,k+1}$. Then a longest cycle
in $G$ is of length $2m-2k-2$ and $e(G)=(m-k-1)n+k+1$.
This tells us $ex(m,n,C_{2t})\geq (m-k-1)n+k+1$. An
immediate consequence of Theorem \ref{Thm-general} is that
$ex(m,n,C_{2t})\leq (m-k-1)n+k+1$. Thus, we have the following
result, which is equivalent to Theorem \ref{Thm-SharpenJackson} and
confirms Conjecture \ref{Conj-Cons}.
\begin{cor}
For any positive integers $m,n,t$, if $n\geq m\geq t\geq
\frac{m}{2}+1$, then
$$ex(m,n,C_{2t})=(t-1)n+m-t+1.$$
\end{cor}

The basic case of the proof of Theorem \ref{Thm-general} is
that the bipartite graph is balanced.
For this special case, a slightly stronger theorem will be proved.

\begin{thm}\label{Thm-balanced}
Let $G$ be a balanced bipartite graph of order $2n$, where $n\geq2k+2$,
$k\in \mathbb{N}$. If $e(G)\geq(n-k-1)n+k+2$,
then there hold:\\
(i) The circumference $c(G)\geq 2n-2k$; and
(ii) $G$ contains cycles of all even lengths from 4 to $c(G)$.
\end{thm}

Our proof of Theorem \ref{Thm-balanced} is motivated
by a theorem of Bondy \cite{B71} stated as follows,
which extends the celebrated Erd\H{o}s-Gallai Theorem \cite{EG59}
on cycles.
\begin{thm}[Bondy \cite{B71}]\label{Thm-Bondy}
Let $G$ be a graph on $n$ vertices and $C$ a longest cycle of $G$
with order $c$. Then
$$e(G-C)+e(G-C,C)\leq \frac{(n-c)c}{2}.$$
\end{thm}

Since $G[C]$ contains at most $\binom{c}{2}$ edges, it can imply Erd\H{o}s-Gallai
Theorem. In fact, Bondy's theorem and its variants turned out to be powerful tools
for tacking many problems on long cycles. For example, it actually plays an
important role in Bollob\'{a}s and Thomason's almost proof \cite{BT99} of Brandt's
conjecture \cite{Br97}, which says that every non-bipartite graph on $n$ vertices
is weakly pancyclic if $e(G)\geq\lfloor\frac{n^2}{4}\rfloor-n+5$. The other example
is that, Bondy's theorem is related to a conjecture of Woodall \cite{W76} in 1976.
Ma and one of authors here \cite{MN} recently proved a stability version of Bondy's
theorem, which is one step towards obtaining a stability version of Woodall's
conjecture \cite{W76}.

Very importantly for us, using Theorem \ref{Thm-Bondy} is an ingenious idea in Bondy's proof of
Tur\'an numbers of large cycles \cite{B71}.
We shall prove a bipartite analog of Theorem \ref{Thm-Bondy} and use it to prove
Theorem \ref{Thm-balanced}.

\begin{thm}\label{Thm-BipartiteBondy}
Let $G=(X,Y;E)$ be a bipartite graph and $C$ a longest cycle of
$G$. Suppose that $|X|=n$, $|Y|=m$, and $|V(C)|=2t$, where $n\geq m\geq t$.
Then there hold:\\
(1) If $m\leq 2t$, then $e(G-C)+e(G-C,C)\leq t(n-1-t)+m$.\\
(2) If $m\geq 2t$, then $e(G-C)+e(G-C,C)\leq t(m+n+1-3t)$.
\end{thm}

The bounds in Theorem \ref{Thm-BipartiteBondy} are tight.
We postpone the discussion to Section 2. Moreover,
Theorem \ref{Thm-BipartiteBondy} generalizes
Theorem \ref{Thm-Jackson}
in the other direction.

Let us digest some notation and terminologies. Let $G=(X,Y;E)$ be a
bipartite graph, where $X,Y$ are two bipartite sets and $E$ is the
edge set of $G$. We say that $G$ is \emph{balanced} if $|X|=|Y|$.
Let $P$ be a path of $G$. We say that $P$ is an $(x,y)$-path if
$x,y$ are two end-vertices of $P$; and $P$ is an $x$-path if $x$ is
one end-vertex of $P$. A graph $G$ is called \emph{weakly pancyclic}
if $G$ contains all cycles of lengths from $g(G)$ to $c(G)$, where
$g(G)$ and $c(G)$ are its girth and circumference, respectively. A
balanced bipartite graph $G$ is called \emph{bipancyclic}, if $G$
contains all cycles of even lengths from 4 to $2|X|$. For a subgraph
$H$ of $G$, we set $X_H=X\cap V(H)$ and $Y_H=Y\cap V(H)$. We use
$|H|$ to denote the order of $H$, that is, $|H|:=|V(H)|$. Let
$S\subseteq V(G)$. We use $G[S]$ to denote the subgraph induced by
$S$, and $G-S$ the subgraph induced by $V(G-S)$. Specially, when
there is no danger of ambiguity, we use $G-H$ instead of $G-V(H)$
sometimes. For $V_1,V_2\subseteq V(G)$ with $V_1\cap
V_2=\emptyset$, we set $e(V_1,V_2)=\{v_1v_2: v_1\in V_1,v_2\in
V_2\}$. If $H_1,H_2$ are two disjoint subgraphs of $G$, then we set
$e(H_1,H_2)=e(V(H_1),V(H_2))$.

The remainder of this paper is organized as follows.
In Section 2, we aim to prove Theorem \ref{Thm-BipartiteBondy}.
In Subsection 2.1, we prove several technical lemmas and
list useful theorems. In Subsection 2.2, we prove Theorem
\ref{Thm-BipartiteBondy}. In Section 3, we prove
Theorem \ref{Thm-balanced} and Theorem \ref{Thm-general}.
\section{Proof of Theorem \ref{Thm-BipartiteBondy}}
\subsection{Preliminaries for proving Theorem \ref{Thm-BipartiteBondy}}
In this subsection, we collect and establish several lemmas to be
used later. Let $G$ be a graph and $P$ be a path of $G$ with the
origin $x$ and terminus $y$. The path $P$ is called a \emph{maximal
path} of $G$, if $N_G(x)\cup N_G(y)\subseteq V(P)$. We say that $P$
is a \emph{maximal $x$-path} of $G$ if $N_G(y)\subseteq V(P)$.

\begin{lem}\label{Lem-specivertexpath}
Let $G=(X,Y;E)$ be a connected bipartite graph and $d(x)\geq d$ for every
vertex $x\in X$.\\
(1) If $|X|\geq |Y|$, then for every vertex $y_0\in Y$, $G$ has a maximal
$y_0$-path with the terminus in $X$ and of order at least $2d$.\\
(2) If $|X|>|Y|$, then for every vertex $x_0\in X$, $G$ has a
maximal $x_0$-path with the terminus in $X$ and of order at least
$2d+1$.
\end{lem}

\begin{proof}
We first show the existence of a maximal $x_0$- or $y_0$-path with
the terminus in $X$. We use induction on $n:=|V(G)|$. The assertion is trivial if
$n=1,2$. Suppose that $n\geq 3$.

First assume $|X|\geq |Y|$ and let $y_0\in Y$.
Thus, $|X|>|Y\backslash\{y_0\}|$. It follows that
there is a component $H$ of $G-y_0$ (possibly $H=G-y_0$) such that
$|X_H|>|Y_H|$. Since $G$ is connected, $y_0$ has a neighbor $x_0\in
X_H$. By the induction hypothesis, $H$ has a maximal $x_0$-path
$P_0$ with the terminus in $X_H\subseteq X$. Thus, $P=y_0x_0P_0$ is a
maximal $y_0$-path with the terminus in $X$.

Now assume $|X|>|Y|$ and let $x_0\in X$. Thus, $|X\backslash\{x_0\}|\geq|Y|$.
It follows that there is a component $H$ of $G-x_0$ (possibly $H=G-x_0$)
such that $|X_H|\geq|Y_H|$. Since $G$ is connected, $x_0$ has a neighbor
$y_0\in Y_H$. By the induction hypothesis, $H$ has a maximal
$y_0$-path $P_0$ with the terminus in $X_H\subseteq X$. Thus,
$P=x_0y_0P_0$ is a maximal $x_0$-path with the terminus in $X$.

It remains to show that the path $P$ is of order at least $2d$
(if $P$ originates at $y_0$), or at least $2d+1$ (if $P$ originates at $x_0$).
Let $x_1\in X$ be the terminus of $X$ other than $x_0$ or $y_0$. Notice that
$d(x_1)\geq d$ and $N(x_1)\subseteq V(P)$. We have $|Y\cap V(P)|\geq d$,
implying that $P$ has order at least $2d$ when $P$ originates at $y_0$, and at
least $2d+1$ when $P$ originates at $x_0$. This proves Lemma \ref{Lem-specivertexpath}.
\end{proof}

For a graph $G$ and $S\subseteq V(G)$, we denote by $\rho_G(S)$
the number of edges in $G$ which are incident to at least one vertex in $S$, that is,
$$\rho_G(S):=e(G[S])+e_G(S,V(G)\backslash S).$$ From this definition, one
can see $d(v)=\rho_G(\{v\})$ for any vertex $v\in V(G)$. When there
is no danger of ambiguity, we use $\rho(u,v)$ and $\rho(S)$ instead
of $\rho_G(\{u,v\})$ and $\rho_G(S)$, respectively. An
$\{s,s'\}$-\emph{disjoint path pair} of $G$ (or shortly, an
$\{s,s'\}$-\emph{DPP}), is the union of an $s$-path and an $s'$-path
which are vertex-disjoint. Let $D$ be an $\{s,s'\}$-DPP, and $t,t'$
be the termini of the two paths in $D$. We say that $D$ is a
\emph{maximal $\{s,s'\}$-DPP} in $G$ if $N_G(t)\cup N_G(t')\subseteq
V(D)$. Clearly, $D$ is a maximal $\{s,s'\}$-DPP of $G$, if and only
if $D+ss'$ is a maximal path of $G+ss'$. For a special case that $G$
is bipartite, we say that $D$ is \emph{detached} if $t$ and $t'$ are
in distinct partition sets of $G$.

Next we shall prove two lemmas on degree conditions
for detached maximal DDP in bipartite graphs.
\begin{lem}\label{Lem-goodmaximalpair}
Let $G=(X,Y;E)$ be a connected balanced bipartite graph. If
$\rho(x,y)\geq \rho$ for every $(x,y)\in(X,Y)$, then for any
$(x_0,y_0)\in(X,Y)$, $G$ has a detached maximal $\{x_0,y_0\}$-DPP of
order at least $\rho+1$.
\end{lem}

\begin{proof}
We first show the existence of the detached maximal
$\{x_0,y_0\}$-DPP by induction on $n:=|V(G)|$. The assertion is
trivial if $n=2$. So assume that $n\geq 4$. Let $x_0\in X,y_0\in
Y$ and let $G':=G-\{x_0,y_0\}$.

First assume that there is a balanced component $H$ of $G'$ that is
incident to both $x_0$ and $y_0$. Let $x_1,y_1\in V(H)$ be the
neighbors of $y_0$ and $x_0$, respectively. By the induction
hypothesis, $H$ has a detached maximal $\{x_1,y_1\}$-DPP, say $D_0$.
Thus, $D=D_0\cup\{x_0y_1,x_1y_0\}$ is a detached maximal
$\{x_0,y_0\}$-DPP of $G$.

Now assume that every balanced component of $G'$ is incident to
either $x_0$ or $y_0$ but not both. Let $\mathcal{H}_1$ be the set
of components $H$ of $G'$ such that either $|X_H|>|Y_H|$, or $H$ is
balanced and incident to $x_0$. Let $\mathcal{H}_2$ be the set of
components $H$ of $G'$ such that either $|Y_H|>|X_H|$, or $H$ is
balanced and incident to $y_0$.

If $y_0$ is not incident to any component of $G'$, then every
component of $G'$ is incident to $x_0$. This fact implies that
$\mathcal{H}_1\neq\emptyset$. Let $H\in\mathcal{H}_1$ and $y_1\in
N_H(x_0)$. By Lemma \ref{Lem-specivertexpath}(1), $H$ has a maximal
$y_1$-path $P_1$ with the terminus in $X_H\subseteq X$. Thus, the union
of the path $P_x=x_0y_1P_1$ and the trivial path $P_y=y_0$ is a
detached maximal $\{x_0,y_0\}$-DPP of $G$, and we are done.

In the following, we assume $y_0$ is incident to at least one
component of $G'$; and similarly, by symmetry, $x_0$ is incident
to at least one component of $G'$.

If $\mathcal{H}_2=\emptyset$, then every component of $G'$ is
balanced, and it follows that $y_0$ is not incident to any component
of $G'$, a contradiction. So $\mathcal{H}_2\neq\emptyset$, and
similarly, $\mathcal{H}_1\neq\emptyset$.

It follows that there exist $H_1\in\mathcal{H}_1$ and
$H_2\in\mathcal{H}_2$ such that: $x_0$ is incident to one of
$H_1$ and $H_2$, and $y_0$ is incident to the other. If $x_0$ is incident
to $H_1$ and $y_0$ is incident to $H_2$, then let $y_1\in
N_{H_1}(x_0)$ and $x_1\in N_{H_2}(y_0)$. Recall that
$|X_{H_1}|\geq|Y_{H_1}|$ and $|Y_{H_2}|\geq|X_{H_2}|$. By Lemma
\ref{Lem-specivertexpath}, $H_1$ has a maximal $y_1$-path $P_1$ with the
terminus in $X_{H_1}\subseteq X$ and $H_2$ has a maximal $x_1$-path
$P_2$ with the terminus in $Y_{H_2}\subseteq Y$. Thus, the union of the
two paths $P_x=x_0y_1P_1$ and $P_y=y_0x_1P_2$ is a detached maximal
$\{x_0,y_0\}$-DPP of $G$. If $x_0$ is incident to $H_2$ and $y_0$ is
incident to $H_1$, then let $y_1\in N_{H_2}(x_0)$ and $x_1\in
N_{H_1}(y_0)$. Note that in this case $|X_{H_1}|>|Y_{H_1}|$ and
$|Y_{H_2}|>|X_{H_2}|$. By Lemma \ref{Lem-specivertexpath}, $H_2$ has
a maximal $y_1$-path $P_1$ with the terminus in $Y_{H_2}\subseteq Y$, and
$H_1$ has a maximal $y_1$-path $P_2$ with the terminus in
$X_{H_1}\subseteq X$. Thus, the union of the two paths
$P_x=x_0y_1P_1$ and $P_y=y_0x_1P_2$ is a detached maximal
$\{x_0,y_0\}$-DPP of $G$. This proves the existence of the detached
maximal $\{x_0,y_0\}$-DPP of $G$.

Now let $D$ be a detached maximal $\{x_0,y_0\}$-DPP of $G$. We will
show that $D$ has order at least $\rho+1$. Let $x_1\in X,y_1\in Y$
be the termini of the two paths in $D$. Obviously, we have
$$\rho(x_1,y_1)=\left\{\begin{array}{ll}
  d(x_1)+d(y_1),    & x_1y_1\notin E(G);\\
  d(x_1)+d(y_1)-1,  & x_1y_1\in E(G).
\end{array}\right.$$
If $x_1y_1\notin E(G)$, then $|V(D)|\geq d(x_1)+d(y_1)+2\geq\rho+1$;
if $x_1y_1\in E(G)$, then $|V(D)|\geq d(x_1)+d(y_1)\geq\rho+1$. This
proves Lemma \ref{Lem-goodmaximalpair}.
\end{proof}

Let $G$ be a graph with connectivity 1, and $u,v\in V(G)$. We call
$\{u,v\}$ a \emph{good pair} of $G$, if there is an end-block $B$ of
$G$ such that exactly one of $u,v$ is an inner-vertex of $B$.

\begin{lem}\label{LeGoodPair}
Let $G=(X,Y;E)$ be a balanced bipartite graph with connectivity 1,
and $\{x_0,x'_0\}\subseteq X$ be a good pair of $G$. Suppose
$\rho(x,y)\geq \rho$ for every $(x,y)\in (X,Y)$. Then $G$ has a
detached maximal $\{x_0,x'_0\}$-DPP of order at least $\rho+1$.
\end{lem}

\begin{proof}
We prove the assertion by induction on $|V(G)|$. It is easy to check
that the assertion is true for $n=4$. Now assume that $n\geq 6$. If
$G$ has a detached maximal $\{x_0,x'_0\}$-DPP, say $D$, then $D$ has
order at least $\rho+1$ (see the last paragraph of the proof of
Lemma \ref{Lem-goodmaximalpair}), and the statement holds. Now we
assume that $G$ has no detached maximal $\{x_0,x'_0\}$-DPP. Let $B$
be an end-block of $G$ that contains one vertex, say $x'_0$, as an
inner-vertex, and let $u$ be the cut-vertex of $G$ contained in $B$.
If $x_0$ is also an inner-vertex of an end-blocks, say $B_0$, then
we assume without loss of generality that $|V(B)|\leq|V(B_0)|$.

Suppose first that $x_0$ is a cut-vertex of $G$. Let $H$ be the
component of $G-x_0$ containing $x'_0$.

If $H$ is balanced, then let $y_0$ be a neighbor of $x_0$ in $H$. By
Lemma \ref{Lem-goodmaximalpair}, $H$ contains an detached maximal
$(x'_0,y_0)$-DPP, say $D$. It follows that $D\cup\{x_0y_0\}$ is an
detached maximal $(x_0,x'_0)$-DPP of $G$, a contradiction. So we
assume that $H$ is not balanced.

If $|X_H|>|Y_H|$, then $|X_{G-H}|<|Y_{G-H}|$. By Lemma
\ref{Lem-specivertexpath}, $H$ has a maximal $x'_0$-path with
terminus in $X_H$ and $G-H$ has a maximal $x_0$-path with terminus
in $Y_{G-H}$. Thus the union of such two paths form an detached
maximal $(x_0,x'_0)$-DPP of $G$, a contradiction. If $|X_H|<|Y_H|$,
then $|X_{G-H}|>|Y_{G-H}|$. By Lemma \ref{Lem-specivertexpath}, $H$
has a maximal $x'_0$-path with terminus in $Y_H$ and $G-H$ has a
maximal $x_0$-path with terminus in $X_{G-H}$. Thus the union of
such two paths form a detached maximal $(x_0,x'_0)$-DPP of $G$, also
a contradiction.

Now we assume that $G'=G-x_0$ is connected. Specially we have
$x_0\in V(G-B)$. Here we deal with the case that $N(x_0)=\{u\}$. For
this case $x_0$ is an inner-vertex of the end-block $B_0$ with
$V(B_0)=\{x_0,u\}$. It follows that $V(B)=\{x'_0,u\}$. By Lemma
\ref{Lem-specivertexpath}, $G-\{x_0,x'_0\}$ has a maximal $u$-path
$P$ with terminus in $Y$. Now the two paths $P_1=x_0uP$ and
$P_2=x'_0$ form a detached maximal $(x_0,x'_0)$-DPP of $G$, a
contradiction. So we assume that $x_0$ has a neighbor $y_0\in
V(G-B)$. We choose $y_0$ such that the distance between $y_0$ and
$u$ in $G$ is as large as possible. It follows that $B$ is an
end-block of $G'$ as well.

Let $H$ be the component of $G'-y_0$ containing $x'_0$. So $B$ is
contained in $H$.

We claim that $B$ is an end-block of $H$ as well. Suppose not. Then
$B=H$. This implies that $N_H(y_0)=\{u\}$, specially $u\in X$. If
$x_0$ has a second neighbor $y_1$, then the distance between $u$ and
$y_1$ is larger than that between $u$ and $y_0$, a contradiction. It
follows that $N(x_0)=\{y_0\}$. Now $x_0$ is an inner-vertex of the
end-block $B_0$ with $V(B_0)=\{x_0,y_0\}$, which contradicting our
choice of $B$. Thus as we claimed, $B$ is an end-block of $H$ as
well.

If $H$ is balanced, then let $x_1=N_H(y_0)$. Since $y_0\in V(G-B)$,
we have that $\{x_1,x'_0\}$ is a good pair of $H$. By induction, $H$
has a detached maximal $(x_1,x'_0)$-DPP $D$. Now
$D\cup\{x_0y_0,y_0x_1\}$ is a detached maximal $(x_1,x'_0)$-DPP of
$G$, a contradiction. So we assume that $H$ is not balanced.

If $|X_H|>|Y_H|$, then $|X_{G-H}|<|Y_{G-H}|$. By Lemma
\ref{Lem-specivertexpath}, $H$ has a maximal $x'_0$-path $P_1$ with
terminus in $X_H$ and $G-H$ has a maximal $y_0$-path $P_2$ with
terminus in $Y_{G-H}$. If $|X_H|<|Y_H|$, then
$|X_{G-H}|\geq|Y_{G-H}|$. By Lemma \ref{Lem-specivertexpath}, $H$
has a maximal $x'_0$-path $P_1$ with terminus in $Y_H$ and $G-H$ has
a maximal $y_0$-path with terminus in $X_{G-H}$. In both cases, the
two paths $P_1$ and $x_0y_0P_2$ form a detached maximal
$(x_0,x'_0)$-DPP of $G$, a contradiction.
\end{proof}

\begin{lem}[Jackson \cite{J85}]\label{ThJa}
Let $G=(X,Y,E)$ be a 2-connected bipartite graph, where
$|X|\geq|Y|$. If each vertex of $X$ has degree at least $k$, and
each vertex of $Y$ has degree at least $l$, then $G$ contains a
cycle of length at least $2\min\{|Y|,k+l-1,2k-2\}$. Moreover, if
$k=l$ and $|X|=|Y|$, then $G$ contains a cycle of length at least
$2\min\{|Y|,2k-1\}$.
\end{lem}

\begin{lem}[Bagga, Varma \cite{BaVa}]\label{ThBaVa}
Let $G=(X,Y,E)$ be a balanced bipartite graph of order $2n$. If
$d(x)+d(y)\geq n+2$ for every $(x,y)\in(X,Y)$, then $G$ is
Hamilton-biconnected.
\end{lem}

\begin{lem}\label{LePathOrder}
Let $G$ be a 2-connected balanced bipartite graph such that
$\rho(x,y)\geq \rho$ for every $(x,y)\in (X,Y)$. Then for any two
vertices $x_1,x_2\in X$, $G$ has an $(x_1,x_2)$-path of order at
least $\rho$.
\end{lem}

\begin{proof}
The assertion can be checked easily when $|X|=2$. So we assume that
$|X|\geq 3$.

Set $k=\min\{d(x): x\in X\}$ and $l=\min\{d(y): y\in Y\}$. It
follows that $k+l=d(x)+d(y)\geq\rho(x,y)\geq\rho$ for some
$(x,y)\in(X,Y)$. If $k\neq l$, then $2k-2\geq\rho-1$ or
$2l-2\geq\rho-1$; if $k=l$, then we have $2k-1\geq\rho-1$. Notice
that $|X|=|Y|$. By Lemma \ref{ThJa}, $G$ has a cycle of length at
least $2\min\{|X|,\rho-1\}$.

Suppose first that $|X|\geq\rho-1$. It follows that $G$ has a cycle
$C$ of length at least $2\rho-2$. Since $G$ is 2-connected, there
are two disjoint paths $P_1,P_2$ between $x_1,x_2$, respectively,
and $C$. Let $u_1,u_2\in V(C)$ be the terminus of $P_1,P_2$
(possibly $x_i=u_i$ for some $i=1,2$). Now one of the two paths
$\overrightarrow{C}[u_1,u_2]$ and $\overleftarrow{C}[u_1,u_2]$ has
order at least $\rho$. Together with $P_1,P_2$, we can get an
$(x_1,x_2)$-path of order at least $\rho$.

Secondly, we suppose that $|X|=\rho-2$. It follows that $G$ has a
Hamiltonian cycle $C$. If one of the two paths
$P_1=\overrightarrow{C}[x_1,x_2]$ and
$P_2=\overleftarrow{C}[x_1,x_2]$ has order at least $\rho$, then
there are noting to prove. So we assume that both $P_1$ and $P_2$
has order exactly $\rho-1$. If there is an edge, say $u_1u_2$, with
$u_i\in V(P_i)\backslash\{x_1,x_2\}$, then one of the paths
$x_1P_1u_1u_2P_2x_2$ and $x_1P_2u_2u_1P_1x_2$ has order at least
$\rho$ (notice that the sum of the orders of such two paths is
$2\rho$), and we are done. Now we assume that there are no edges
between $V(P_1)\backslash\{x_1,x_2\}$ and
$V(P_2)\backslash\{x_1,x_2\}$. It follows that for any two vertices
$(x,y)\in(X\backslash\{x_1,x_2\},Y)$, $d(x)+d(y)\leq\rho-1$, a
contradiction.

Lastly, we suppose that $|X|\leq\rho-3$. Let $y_1\in N(x_1)$, $y_2\in
Y\backslash\{y_1\}$, and set $G'=G-\{x_1,y_2\}$. Now $G'$ is a
balanced bipartite graph of order $2(\rho-4)$ and for every
$(x,y)\in(X_{G'},Y_{G'})$,
$d_{G'}(x)+d_{G'}(y)\geq\rho_{G'}(x,y)\geq\rho-2$. By Lemma
\ref{ThBaVa}, $G'$ is Hamilton-biconnected. Let $P'$ be a
Hamiltonian $(y_1,x_2)$-path of $G'$. Then $P=x_1y_1P'$ is an
$(x_1,x_2)$-path of order $|V(G)|-1$. Notice that if $G$ is complete and
bipartite, then $\rho\leq|V(G)|-1$; otherwise
$\rho\leq|V(G)|-2$. It follows that $P$ is an $(x_1,x_2)$-path of
order at least $\rho$.
\end{proof}

A subgraph $F$ of a graph $G$ is called an \emph{$(x,L)$-fan} if $F$
has the following decomposition $F=\cup_{i=1}^kP_i$, where
\begin{itemize}
\item $k\geq 2$;
\item each $P_i$ is a path with two end-vertices $x$ and $y_i\in V(L)$;
\item $V(P_i)\cap V(L)=\{y_i\}$ and $V(P_i)\cap V(P_j)=\{x\}$.
\end{itemize}
The following lemma on the fan structure is a corollary of a theorem on weighted
graphs, which was proved by Fujisawa, Yoshimoto, and Zhang (see \cite[Lemma~1]{FYZ05}).
We need this refined version of the Fan Lemma to find a long cycle.
\begin{lem}\label{LemFYZ}
Let $G$ be a 2-connected graph, $C$ a longest cycle $G$, and $H$ a
component of $G-C$. If $d(v)\geq d$ for every $v\in V(H)$, then for
every vertex $x\in V(H)$, there is an $(x,C)$-fan $F$ with $e(F)\geq
d$.
\end{lem}

We also need the following two results on long cycles in bipartite graphs
due to Jackson \cite{J81,J85}.
\begin{lem}\rm ({Jackson \cite[Lemma~5]{J85}})\label{Lem-Jacksonmaximalpath}
Let $G=(X,Y;E)$ be a 2-connected bipartite graph. Let $P$ be a maximal
path in $G$ with two end-vertices $u$ and $v$.\\
(i) If $u\in X$ and $v\in Y$, then $G$ contains a cycle of length at
least $\min\{|V(P)|,2(d(u)+d(v)-1)\}$.\\
(ii) If $u,v\in X$ then $G$ contains a cycle of length at least
$\min\{|V(P)|-1,2(d(u)+d(v)-2)\}$.
\end{lem}
The final lemma was originally conjectured by Sheehan (see \cite[pp.332]{J81}).
\begin{lem}{\rm (Jackson \cite[Theorem~1]{J81})}\label{Lem_Cyclecontaining}
Let $G=(X,Y;E)$ be a bipartite graph with $|X|\leq |Y|$. If $d(x)\geq\max\{|X|,|Y|/2+1\}$ for every
vertex $x\in X$, then $G$ has a cycle containing all vertices in $X$.
\end{lem}

\subsection{Proof of Theorem \ref{Thm-BipartiteBondy}}
We first discuss on the extremal graphs of Theorem \ref{Thm-BipartiteBondy}.
Before the discussion, let us introduce some notations.

Let $a$, $b$, and $c$ be three positive integers. Let $\mathcal{B}_{a,b}$ be the set of bipartite
graphs with two partition sets of size $a$ and $b$, respectively. We define a graph $L_{a,b}^c$
as follows. (If the notation $c$ is unemphatic, we use $L_{a,b}$ instead.)

\noindent
$\bullet$ If $a\leq c$ or $b\leq c$, then let $L_{a,b}=K_{a,b}$.\\
$\bullet$ If $c<b\leq \max\{a,2c\}$, then let $L_{a,b}$ be the
graph by identifying one vertex from the $a$-set of $K_{a,c}$ and
the other one from the $1$-set of $K_{1,b-c}$.\\
$\bullet$ If $c< a\leq \max\{b,2c\}$, then let $L_{a,b}$ be the
graph by identifying one vertex from the $b$-set of $K_{b,c}$ and
the other vertex from the $1$-set of $K_{1,a-c}$.\\
$\bullet$ If $2c<\max\{a,b\}$, then let $L_{a,b}$ be the
graph by identifying one vertex from the $c$-set of $K_{c,b-c}$ and
the other vertex from the $(a-c+1)$-set of $K_{a-c+1,c}$.

By Theorem \ref{Thm-Jackson}, $L_{a,b}^c$ is a graph in $\mathcal{B}_{a,b}$ with the maximum number
of edges among those without cycles of length more than $2c$ (see Jackson \cite{J85}). One
can see the graph $L_{a,b}^c$ shows that the bounds in Theorem \ref{Thm-BipartiteBondy}
are tight for each case.

We define $$\varrho(a,b)=e(L_{a,b})-c^2.$$ Notice that if $C$ is a
longest cycle of a graph $G=L_{a,b}^c$, $c\leq b\leq a$, then
$\varrho(a,b)=\rho(G-C)$.

Armed with the necessary additional notations, let us restate
Theorem \ref{Thm-BipartiteBondy} as follows.

\smallskip
\noindent
{\bf Theorem $1.7^{'}$.} Let $G=(X,Y;E)$ be a bipartite graph,
where $|X|=a\geq b=|Y|$. Let $C$ be a longest cycle of
$G$ of length $2c$.
Then $\rho(G-C)\leq\varrho(a,b)$, i.e.,\\
(1) if $b\leq 2c$, then $\rho(G-C)\leq c(a-1-c)+b$; and\\
(2) if $b\geq 2c$, then $\rho(G-C)\leq c(a+b+1-3c)$.

Now we give a proof of Theorem $1.7^{'}$.

\smallskip
\noindent
{\bf Proof of Theorem $1.7^{'}$.}
We prove the theorem by contradiction. Let $G$ be a counterexample
to Theorem $1.7^{'}$ such that:\\
(i) $|G|$ is minimum; and\\
(ii) subject to (i), $e(G)$ is maximum.

\begin{claim}\label{Claim2-connected}
$G$ is 2-connected.
\end{claim}

\begin{proof}
If $G$ is disconnected, then let $G_1$ be a connected bipartite
graph obtained from $G$ by adding $\omega(G)-1$ edges such that each
new edge is between distinct partition sets, where $\omega(G)$
denotes the number of components in $G$. Since the added edges are
cut-edges of $G_1$, $C$ is a longest cycle of $G_1$ as well and all
the add edges are outside $C$. So $\rho_G(G-C)<\rho_{G_1}(G_1-C)$,
but $e(G_1)>e(G)$, a contradiction to (ii). This implies $G$ is
connected. Suppose now that $G$ has connectivity 1.
Let $B$ be an end-block of $G$ with smallest order
among those not containing $C$. We choose $G$
such that:\\
(iii) subject to (i),(ii), $|B|$ is minimum.

Let $u_0$ the cut-vertex of $G$ contained in $B$. Set
$$\theta_{X}=\left\{\begin{array}{ll}
  1,   &u_0\in X;\\
  0,   &u_0\in Y.
\end{array}\right. \mbox{ and~~~ } \theta_{Y}=\left\{\begin{array}{ll}
  0,    & u_0\in X;\\
  1,    & u_0\in Y.
\end{array}\right.$$

We first claim that $|X_B|\geq 2$ and $|Y_B|\geq 2$. Suppose not.
Since $B$ is bipartite and non-separable, we deduce $B\cong K_2$.
Set $V(B)=\{u_0,v_0\}$. So $v_0$ is of degree 1. Let $G_2:=G-v_0$.
By the choice of $G$, we get
$$\rho(G-C)=\rho_{G_2}(G_2-C)+1\leq\varrho(a-\theta_X,b-\theta_Y)+1\leq\varrho(a,b),$$
a contradiction. Thus, $|X_B|\geq 2$ and $|Y_B|\geq 2$.

Let $u_3\in V(C)$ such that $u,u_3$ are in the same partition set
(possibly $u_0=u_3$). Let
$$G_3:=(G-E(u_0,B-u_0))\cup\{u_3v: v\in V(B), u_0v\in E(G)\}.$$
Clearly, $C$ is a longest cycle of $G_3$ as well, and
$\rho(G-C)=\rho_{G_3}(G_3-C)$. Hence $G_3$ is also a counterexample
satisfying (i)(ii)(iii).

We use $B_3$ to denote the end-block of $G_3$ with the vertex set
$(V(B)\backslash\{u_0\})\cup \{u_3\}$. Set $D_3=G_3-(B_3-u_3)$. So
$G_3$ consists of $B_3$ and $D_3$. We have
$|X_{B_3}|+|X_{D_3}|=a+\theta_X$, and
$|Y_{B_3}|+|Y_{D_3}|=b+\theta_Y$. Since $D_3$ contains the cycle
$C$, $|X_{D_3}|\geq c$ and $|Y_{D_3}|\geq c$.

Let $B_4$ be a graph on $V(B_3)$ isomorphic to
$L_{|X_{B_3}|,|Y_{B_3}|}$ (recall the definition), $D_4$ a graph on $V(D_3)$ isomorphic
to $L_{|X_{D_3}|,|Y_{D_3}|}$, and $G_4$ the union of $B_4$ and
$D_4$. Clearly, the longest cycle of $D_4$ is of length $2c$. We
choose $D_4$ such that $C$ is a (longest) cycle in $D_4$ as well.
Also note that $B_3$ has no cycle of length more than $2c$. By
the choice of $G$,
$$\rho_{D_3}(D_3-C)\leq\varrho(|X_{D_3}|,|Y_{D_3}|)=\rho_{D_4}(D_4-C),$$
and furthermore, by Theorem \ref{Thm-Jackson},
$$e(B_3)\leq e(L_{|X_{B_3}|,|Y_{B_3}|})=e(B_4).$$
Thus, we have
$$\rho_{G_3}(G_3-C)=\rho_{D_3}(D_3-C)+e(B_3)\leq\rho_{D_4}(D_4-C)+e(B_4)=\rho_{G_4}(G_4-C).$$
Since $e(G_4[C])=c^2$ and $e(G_4)\geq e(G_3)$, we can see that $G_4$
is a counterexample satisfying (i)(ii)(iii).

If $B_4$ is separable, then we have a contradiction to (iii). So
$B_4$ is 2-connected, i.e., $B_4$ is an end-block of $G_4$. By the
definition of $B_4\cong L_{|X_{B_3}|,|Y_{B_3}|}$, we infer
$\min\{|X_{B_4}|,|Y_{B_4}|\}\leq c$ and $B_4\cong
K_{|X_{B_3}|,|Y_{B_3}|}$. Recall that $u_3$ is the cut-vertex of
$G_4$ contained in $B_4$. Let $u_4$ be a vertex in
$V(B_4)\backslash\{u_3\}$ with $d_{B_4}(u_4)\leq c$ (recall that
$|X_{B_4}|\geq 2$ and $|Y_{B_4}|\geq 2$). Set
$$\vartheta_x=\left\{\begin{array}{ll}
  1,  & u_4\in X;\\
  0,  & u_4\in Y.
\end{array}\right. \mbox{ and } \vartheta_y=\left\{\begin{array}{ll}
  0,    & u_4\in X;\\
  1,    & u_4\in Y.
\end{array}\right.$$

We will show that $\min\{|X_{D_4}|,|Y_{D_4}|\}=c$. Suppose that
$\min\{|X_{D_4}|,|Y_{D_4}|\}>c$. Without loss of generality, we assume
$|X_{D_4}|\geq|Y_{D_4}|$ (the other case can be dealt with
similarly). If $c<|Y_{D_4}|\leq 2c$, then $D_4$ has a pendent edge,
a contradiction to (iii). So $|Y_{D_4}|>2c$,
and $D_4$ consists of $A_x\cong K_{c,|Y_{D_4}|-c}$
and $A_y\cong K_{|X_{D_4}|+1-c,c}$, with a common vertex in
$X_{D_4}$. Let $G_5$ be the graph obtained from $G_4-E(u_4,B_4)$ by
adding edges from $u_4$ to all vertices in $X_{A_x}$ (if $u_4\in
Y_{B_4}$) or $Y_{A_y}$ (if $u_4\in X_{B_4}$). In each case, the
vertex $u_4$ is of degree $c$ in $G_5$. That is, $G_5=B_5\cup
D_5$ where
$$B_5=L_{|X_{B_4}|-\vartheta_x,|Y_{B_4}|-\vartheta_y},
D_5=L_{|X_{D_4}|+\vartheta_x,|Y_{D_4}|+\vartheta_y}.$$ Obviously, $C$
is a longest cycle of $G_5$ as well, and $\rho_{G_4}(G_4-C)\leq\rho_{G_5}(G_5-C)$.
Note that the end-block $B_5$ of $G_5$ has order less than $B$, a
contradiction to (iii). Thus, $\min\{|X_{D_4}|,|Y_{D_4}|\}=c$.

By symmetry, we assume that $|X_{D_4}|=c$. If $|X_{B_4}|\leq c$,
then we can get a contradiction
similarly as above. So assume that $|X_{B_4}|>c$, which implies
that $|Y_{B_4}|\leq c$, since $\min\{|X_{B_4}|,|Y_{B_4}|\}\leq c$.

We claim that $|Y_{B_4}|=c$. Suppose that $|Y_{B_4}|<c$. If
$|Y_{D_4}|=c$ as well, then we can get a contradiction similarly as
above. So $|Y_{D_4}|>c$. Let $u_6$ be a vertex in
$Y_{D_4-C}$, and let $G_6$ be the graph obtained from
$G_4-E(u_6,D_4)$ by adding edges from $u_6$ to all vertices in
$X_{B_4}$. Clearly $e(G_4)<e(G_6)$. One can see that
$G_6$ has no cycle of length more than $2c$. Thus, $C$ is a longest
cycle of $G_6$ as well, and $\rho_{G_4}(G_4-C)\leq\rho_{G_6}(G_6-C)$,
a contradiction to (ii). Hence we conclude $|Y_{B_4}|=c$.

This implies that $G_4$ is isomorphic to $L_{a,b}$ or $L_{b,a}$. In
any case, we have $\rho_{G_4}(G_4-C)=\varrho(a,b)$, a contradiction.
This proves Claim \ref{Claim2-connected}.
\end{proof}

Next we distinguish the following cases and derive a contradiction for each one.

\begin{case}
$2c<b\leq a$.
\end{case}

For this case, $\varrho(a,b)=\varrho(a-1,b)+c=\varrho(a,b-1)+c$. If
there is a vertex in $G-C$ with degree at most $c$, then by the
choice of $G$, $\rho_{G}(G-C)\leq\varrho(a,b)$. Thus, we assume that
every vertex in $G-C$ has degree at least $c+1$. Let $H$ be a
component of $G-C$. By Lemma \ref{LemFYZ}, for each vertex $v\in H$,
there is an $(v,C)$-fan $F$ with $e(F)\geq c+1$. This implies
$|V(C)|\geq 2c+2$, a contradiction.

\begin{case}
$c\leq b\leq 2c<a$.
\end{case}

For this case, $\varrho(a,b)=\varrho(a-1,b)+c$. If there is a vertex
in $X_{G-C}$ with degree at most $c$, then we can get a
contradiction by the choice of $G$. So assume that every vertex in
$X_{G-C}$ has degree at least $c+1$. Let $X'\subseteq X_{G-C}$ with
$|X'|=c+1$, and $G'=G[X'\cup Y]$. Observe that for every $x\in X'$,
$d_{G'}(x)\geq c+1=\max\{|X'|,|Y|/2+1\}$. By Lemma
\ref{Lem_Cyclecontaining}, $G'$ has a cycle containing all vertices
in $X'$, i.e., $G$ has a cycle of length $2c+2$, a contradiction.

\begin{case}
  $c\leq b<a\leq 2c$.
\end{case}

For this case, $\varrho(a,b)=\varrho(a-1,b)+c$. So every vertex in
$X_{G-C}$ has degree at least $c+1$. Since $a>b$, we can choose $H$
as a component of $G-C$ with $|X_H|>|Y_H|$. Let
$N_C(H)=\{u_1,u_2,\cdots,u_\alpha\}$, where $u_1,\cdots,u_\alpha$
appear in this order along $C$. Let $v_1\in N_H(u_1)$. By Lemma
\ref{Lem-specivertexpath}, $H$ has a maximal $v_1$-path $P$ with
terminus in $X_H$. Let $s\in X_H$ be the terminus of $P_1$. We
extend the path $sP_1v_1u_1\overleftarrow{C}[u_1,u_2]$ to be a
maximal $t$-path, say $P$. Thus $P$ is a maximal path of $G$. Let
$t$ be the end-vertex of $P$ other than $s$. Since $d(s)\geq c+1$,
we have $|V(P)|\geq 2c+2$ and $d(s)+d(t)\geq c+3$. By Lemma
\ref{Lem-Jacksonmaximalpath}, $G$ has a cycle of length more than
$2c$, a contradiction.

\begin{case}
$c\leq b=a\leq 2c$.
\end{case}

For this case, $\varrho(a,b)=\varrho(a-1,b-1)+c+1$. If there is a
pair of vertices $(x,y)\in(X_{G-C},Y_{G-C})$ with $\rho(x,y)\leq
c+1$, then we can prove by the choice of $G$. So assume that for
every $(x,y)\in(X_{G-C},Y_{G-C})$, $\rho(x,y)\geq c+2$. Recall that
$$\rho(x,y)=\left\{\begin{array}{ll}
  d(x)+d(y),    & xy\notin E(G);\\
  d(x)+d(y)-1,  & xy\in E(G).
\end{array}\right.$$

\begin{subcase}
  $G-C$ is disconnected.
\end{subcase}

First assume that there is a balanced component $H$ of $G-C$. Then
both $G_1=G[V(C)\cup V(H)]$ and $G_2=G-H$ are balanced. Clearly, $C$
is a longest cycle of $G_1$ and $G_2$ as well. Since $|G_1|<|G|$, we
get
$$\rho_{G_1}(G_1-C)\leq c(|X_{G_1}|-1-c)+|Y_{G_1}|=(c+1)|X_H|,$$ and similarly,
$$\rho_{G_2}(G_2-C)\leq(c+1)|X_{G-C-H}|.$$ Thus
$$\rho(G-C)=\rho(H)+\rho(G-C-H)=\rho_{G_1}(G_1-C)+\rho_{G_2}(G_2-C)\leq(c+1)|X_{G-C}|=\varrho(a,b).$$

Now we assume that every component of $G-C$ is not balanced. Let
$H_1, H_2$ be two components of $G-C$ such that
$|X_{H_1}|>|Y_{H_1}|$, $|X_{H_2}|<|Y_{H_2}|$. Let $N_C(H_1)\cup
N_C(H_1)=\{u_1,u_2,\ldots,u_\alpha\}$, where the vertices appear in
this order along $C$. Set $\beta_1=\min\{d_{H_1}(x): x\in X_{H_1}\}$
and $\beta_2=\min\{d_{H_2}(y): y\in Y_{H_2}\}$. Thus
$\alpha+\beta_1+\beta_2\geq c+2$.

Without loss of generality, we suppose $u_1\in N_C(H_1)$ and $u_2\in
N_C(H_2)$. Let $v_1\in N_{H_1}(u_1)$, $v_2\in N_{H_2}(u_2)$. By
Lemma \ref{Lem-specivertexpath}, $H_1$ has a maximal $v_1$-path
$P_1$ with terminus in $X_{H_1}$ and of order at least $2\beta_1$,
$H_2$ has a maximal $v_2$-path $P_2$ with terminus in $Y_{H_2}$ and
of order at least $2\beta_2$. Thus,
$P=P_1v_1\overleftarrow{C}[u_1,u_2]v_2P_2$ is a maximal path of $G$.
Let $s,t$ be the ends of $P$. Recall that $d(s)+d(t)\geq c+2$. If
$|V(P)|\geq 2c+1$, then by Lemma \ref{Lem-Jacksonmaximalpath}, $G$
has a cycle of length more than $2c$, a contradiction. So assume
that $|V(P)|\leq 2c$. This implies that
$\ell(\overrightarrow{C}[u_1,u_2])\geq 2$. Therefore, we infer
$\ell(\overrightarrow{C}[u_i,u_{i+1}])\geq 2$ for all
$i=1,2,\cdots,\alpha$ (the index are taken modular $\alpha$).
Moreover,
$$2c\geq|V(P_1)|+|V(P_2)|+\ell(\overleftarrow{C}[u_1,u_2])+1\geq
2\beta_1+2\beta_2+2c+1-\ell(\overrightarrow{C}[u_1,u_2]).$$ Thus we
have $\ell(\overrightarrow{C}[u_1,u_2])\geq 2\beta_1+2\beta_2+1$. So
$$\ell(C)=\sum_{i=1}^\alpha\ell(\overrightarrow{C}[u_i,u_{i+1}])\geq
2\beta_1+2\beta_2+1+2(\alpha-1)>2c,$$ a contradiction.

\begin{subcase}
  $G-C$ is connected.
\end{subcase}

Let $H=G-C$, $N_C(H)=\{u_1,u_2,\ldots,u_\alpha\}$, and
$\beta=\min\{\rho_H(x,y): x\in X_H, y\in Y_H\}$. So
$\alpha+\beta\geq c+2$. Note that $H$ is balanced.

\begin{subsubcase}
$N_C(H)\cap X_C\neq\emptyset$ and $N_C(H)\cap Y_C\neq\emptyset$.
\end{subsubcase}

Without loss of generality, we assume that $u_1\in X_C$, $u_2\in
Y_C$. Let $v_1\in N_H(u_1)$, $v_2\in N_H(u_2)$. By Lemma
\ref{Lem-goodmaximalpair}, $H$ has a detached maximal
$(v_1,v_2)$-DPP $D$ with order at least $\beta+1$. Thus $P=D\cup
v_1u_1\overleftarrow{C}u_2v_2$ is a maximal path of $G$. Let $s,t$
be the two end-vertices of $P$. So $d(s)+d(t)\geq c+2$. If
$|V(P)|\geq 2c+1$, then by Lemma \ref{Lem-Jacksonmaximalpath}, $G$
has a cycle of length more than $2c$, a contradiction. So we derive
$|V(P)|\leq 2c$. Therefore,
$$2c\geq|V(D)|+\ell(\overleftarrow{C}[u_1,u_2])+1\geq
\beta+1+2c+1-\ell(\overrightarrow{C}[u_1,u_2]).$$ This implies
$\ell(\overrightarrow{C}[u_1,u_2])\geq \beta+2$.

Note that there also exists some subscript $i$ such that $u_i\in
Y_C$ and $u_{i+1}\in X_C$. We can get
$\ell(\overrightarrow{C}[u_i,u_{i+1}])\geq\beta+2$ as above. So
$$\ell(C)=\sum_{i=1}^\alpha\ell(C[u_i,u_{i+1}])\geq
2(\beta+2)+2(\alpha-2)>2c,$$ a contradiction.

\begin{subsubcase}
$N_C(H)\cap X_C=\emptyset$ or $N_C(H)\cap Y_C=\emptyset$.
\end{subsubcase}

Without loss of generality, we assume that $N_C(H)\subseteq Y_C$.
Recall that $N_C(H)=\{u_1,u_2,\ldots,u_\alpha\}$, and
$\beta=\min\{\rho_H(x,y): x\in X_H, y\in Y_H\}$. Since $G$ is
2-connected, $\alpha\geq 2$ and $|N_H(C)|\geq 2$. Specially,
$|X_H|=|Y_H|\geq 2$.

First assume that $H$ is 2-connected. Then at least two of $u_i$ have
the property that $u_i,u_{i+1}$ are adjacent to two distinct
vertices in $H$. If $u_i,u_{i+1}$ are adjacent to two distinct
vertices in $H$, then by Lemma \ref{LePathOrder}, there is a
$(u_i,u_{i+1})$-path of length at least $\beta+1$ with all internal
vertices in $H$. It follows that
$\ell(\overrightarrow{C}[u_i,u_{i+1}])\geq\beta+1$. Recall that
there are at least two pairs $\{u_i,u_{i+1}\}$ that have two
distinct neighbors in $H$. Thus
$$\ell(C)=\sum_{i=1}^\alpha\ell(C[u_i,u_{i+1}])\geq
2(\beta+1)+2(\alpha-2)>2c,$$ a contradiction.

Now we assume that $H$ has connectivity 1. Then at least two of
$u_i$ have the property that $u_i,u_{i+1}$ are adjacent to a good
pair of $H$. If $u_i,u_{i+1}$ are adjacent to a good pair of $H$,
say $\{x_i,x_{i+1}\}$, then by Lemma \ref{LeGoodPair}, $H$ has a
detached maximal $\{x_i,x_{i+1}\}$-DPP $D$ of order at least
$\beta+1$. Let $s,t$ be the termini of the two paths of $D$. It
follows that $P=D\cup(x_1u_1\overleftarrow{C}u_2x_2)$ is a maximal
path of $G$. If $|V(P)|\geq 2c+1$, then $G$ has a cycle of length
more than $2c$ by Lemma \ref{Lem-Jacksonmaximalpath}. So assume that
$|V(P)|\leq 2c$. That is,
$$2c\geq|V(P)|\geq\beta+1+\ell(\overleftarrow{C}[u_i,u_{i+1}])+1=\beta+2c+2-\ell(\overrightarrow{C}[u_i,u_{i+1}]).$$
This implies $\ell(\overrightarrow{C}[u_i,u_{i+1}])\geq\beta+2$.

Recall that there are at least two pairs $\{u_i,u_{i+1}\}$ that
are adjacent to a good pair of $H$. It follows that
$$\ell(C)=\sum_{i=1}^\alpha\ell(C[u_i,u_{i+1}])\geq
2(\beta+2)+2(\alpha-2)>2c,$$ a contradiction. The proof is complete.
\qed

\section{Proofs of Theorems \ref{Thm-general} and \ref{Thm-balanced}}
In this section, we prove Theorem \ref{Thm-balanced}
and then prove Theorem \ref{Thm-general}.
We first give a sketch of the proof of Theorem \ref{Thm-balanced}.
We shall show that $G$ is weakly bipancyclic with girth 4 and prove
it by contradiction. Suppose not. Then for a longest cycle $C$,
the hamiltonian graph $G[C]$ is not bipancyclic. By a theorem of
Entringer and Schmeichel \cite{ES88}, we can get an upper bound
of $e(G[C])$. Notice that Theorem \ref{Thm-BipartiteBondy} gives
an upper bound of $e(G-C)+e(G-C,C)$. Finally, this can give us an
estimate of an upper bound of $e(G)$, which contradicts the edge
number condition.

Proof of Theorem \ref{Thm-balanced} needs the following.

\begin{theoremA}[Entringer and Schmeichel \cite{ES88}]\label{ThES}
Let G be a hamiltonian bipartite graph of order $2n\geq 8$.
If $e(G)>\frac{n^2}{2}$, then $G$ is bipancyclic.
\end{theoremA}

Now we are in stand for giving the proof of Theorem \ref{Thm-balanced}.

\noindent
\smallskip
{\bf Proof of Theorem \ref{Thm-balanced}.} (i) Set $n=m$ and $t=n-k$
in Theorem \ref{Thm-Jackson}. Since $n\geq 2k+2$, we have $n\leq
2t-2$. By computation, we get $e(G)\geq n(n-k-1)+k+2>n+(n-1)(t-1)$.
By Theorem \ref{Thm-Jackson}, $c(G)\geq 2t=2(n-k)$. The proof is
complete.

\noindent
(ii) Suppose that $G$ is not weakly bipancyclic
with girth 4. Let $C$ be a longest cycle in $G$, $2c$ the length of $C$ in $G$,
and $G'=G[C]$. Notice that $G'$ is hamiltonian. Since $G$ is not weakly bipancyclic
with girth 4, $G'$ is not bipancyclic. By Theorem \ref{ThES}, we have $e(G')\leq \frac{c^2-1}{2}$.
By Theorem \ref{Thm-balanced}(i), we know $c\geq n-k\geq k+2$ where $k\in \mathbb{Z}$, and $2c-n\geq n-2k>0$.
By Theorem \ref{Thm-BipartiteBondy},
\begin{align*}
e(G)&=e(G')+e(G-C,C)+e(G-C)\leq -\frac{c^2}{2}+c(n-1)+n-\frac{1}{2}.
\end{align*}
Recall that $n-k\leq c\leq n$. Set a function $f(x)=-\frac{x^2}{2}+x(n-1)+n-\frac{1}{2}$,
where $n-k\leq x\leq n$.
The symmetric axis is $x=n-1$.
First, suppose that $k=0$ or $k\geq 2$. Notice that
$n-(n-1)\leq (n-1)-(n-k)$. There holds
\begin{align*}
f(x)\leq f(n)=\max\{f(n-k),f(n)\}=\frac{n^2}{2}-\frac{1}{2}<n(n-k-1)+k+2,
\end{align*}
since $n\geq 2(k+1)$, a contradiction. Let $k=1$. Then
$f(n-1)=\frac{(n-1)^2}{2}+n-\frac{1}{2}<n(n-2)+3$, a contradiction.
The proof is complete. \qed

\smallskip
Using Theorem \ref{Thm-balanced}, we can prove Theorem \ref{Thm-general}
now.

\smallskip
\noindent
{\bf Proof of Theorem \ref{Thm-general}.}
We prove the theorem by contradiction. Let $G$ be a counterexample such that
the order $n+m$ is smallest, and then the number of edges is smallest.
By Theorem \ref{Thm-balanced}, we know $n\geq m+1$. If there is a vertex,
say $y\in Y$, such that $d_X(y)\leq m-k-1$, then let $G':=G-y$.
We can see $e(G')=e(G)-d_X(y)\geq (n-1)(m-k-1)+k+2$. By the choice of $G$,
$G'$ is not a counterexample, and so is $G$, a contradiction.
Thus, $d_X(y)\geq m-k$ for each vertex $y\in Y$. Hence $e(G)\geq n(m-k)>(n-1)(m-k-1)+k-2$,
since $m\geq 2k+2$ and $n\geq 2$. Now we can delete an edge in $G$ and
find a smaller counterexample, a contradiction. This contradiction completes
the proof.
\qed

\section*{Acknowledgements}
Part of writing was done when the second author
was visiting Fengming Dong at NTU, Singapore.
The second author is grateful to all his help, enthusiasm,
and encouragement during his visit. The second author is also
grateful to Jie Ma for sharing his knowledge on
this topic.


\begin{thebibliography}{10}
\bibitem{A09}
L. Adamus, Edge condition for long cycles in bipartite graphs,
\emph{Discrete Math. Theor. Comput. Sci.}, {\bf 11} (2009), no. 2, 25--32.

\bibitem{BaVa}
K.S. Bagga, B. N. Varma, Bipartite graphs and degree conditions, in:
Graph Theory, Cornbinatorics, Algorithms, and Applications, (ed. Y.
Alavi, F. Chung, R. Graham, D. Hsu, Siam 1991) pp. 564--573.

\bibitem{BGMV07}
C. Balbuena, P. Garc\'{i}a-V\'{a}zquez, X. Marcote, and J.C. Valenzuela,
Counterexample to a conjecture of G\"{yo}ri on $C_{2l}$-free bipartite
graphs, \emph{Discrete Math.}, {\bf 307} (2007), 748--749.

\bibitem{BT99}
B. Bollob\'{a}s and A. Thomason, Weakly pancyclic graphs,
\emph{J. Combin. Theory, Ser. B}, {\bf 77} (1999),
121--137.

\bibitem{B71}
J.A. Bondy, Large cycles in graphs, \emph{Discrete Math.}, {\bf 1} (1971), no. 2, 121--132.

\bibitem{BC76}
J. A. Bondy and V. Chv\'atal, A method in graph theory, \emph{Discrete Math.}, {\bf 15}
(1976), 111--135.

\bibitem{Br97}
S. Brandt, A sufficient condition for all short cycles, \emph{Discrete Appl. Math.}, {\bf 79} (1997), 63--66.

\bibitem{ES88}
R.C. Entringer and E.F. Schmeichel, Edge conditions and cycle structure in bipartite graphs,
\emph{Ars Combin.}, {\bf 26} (1988), 229--232.

\bibitem{EG59}
P.~Erd\H{o}s and T.~Gallai, On maximal paths and circuits of graphs,
\emph{Acta Math. Hungar.}, {\bf 10} (1959), 337--356.

\bibitem{ESS95}
P. Erd\H{o}s, A. S\'{a}rk\"{o}zy, and V.T. S\'{o}s, On product representation of powers I,
\emph{European J. Combin.}, {\bf 16} (1995), 323--331.

\bibitem{FYZ05}
J. Fujisawa, K. Yoshimoto, and S. Zhang, Heavy cycles passing through some
specified vertices in weighted graphs, \emph{J. Graph Theory},
{\bf 49} (2005), no. 2, 93--103.

\bibitem{F96}
Z. F\"{u}redi, New asymptotics for bipartite Tur\'an numbers,
\emph{J. Combin. Theory, Ser. A}, {\bf 75} (1996), no. 1, 141--144.

\bibitem{FS13}
Z. F\"{u}redi and M. Simonovits, The history of degenerate (bipartite) extremal graph problems,
Erd\H{o}s centennial, 169--264, Bolyai Soc. Math. Stud., 25, J\'{a}nos Bolyai Math. Soc., Budapest, 2013.

\bibitem{GRS84}
A. Gy\'{a}rf\'{a}s, C. C. Rousseau, and R. H. Schelp, An extremal problem
for paths in bipartite graphs, \emph{J. Graph Theory}, {\bf 8} (1984), 83--95.

\bibitem{G97}
E. Gy\"{o}ri, $C_6$-free bipartite graphs and product representation
of squares, \emph{Discrete Math.}, {\bf 165/166} (1997), 371--375.

\bibitem{G06}
E. Gy\"{o}ri, Triangle-free hypergraphs,
\emph{Combin. Probab. Comput.}, {\bf 15}(1-2) (2006), 185--191.

\bibitem{J81}
B. Jackson, Cycles in bipartite graphs, \emph{J. Combin. Theory, Ser. B}, {\bf 30}
(1981), 332--342.

\bibitem{J85}
B. Jackson, Long cycles in bipartite graphs, \emph{J. Combin. Theory, Ser. B}, {\bf 38}
(1985), 118--131.

\bibitem{JL94}
B. Jackson and H. Li, Hamilton cycles in $2$-connected regular bipartite graphs,
\emph{J. Combin. Theory, Ser. B}, {\bf 62} (1994), no. 2, 236--258.

\bibitem{JM18}
T. Jiang and J. Ma, Cycles of given lengths in hypergraphs, \emph{J. Combin. Theory,
Ser. B}, {\bf 133} (2018), 54--77.

\bibitem{KV05}
A. Kostochka and J. Verstra\"{e}te, Even cycles in hypergraphs,
\emph{J. Combin. Theory, Ser. B}, {\bf 94} (2005), no. 1, 173--182.

\bibitem{LK18}
A. Kostochka and R. Luo, On $r$-uniform hypergraphs with circumference less than $r$,
arXiv:1807.04683.

\bibitem{KST54}
T. K\H{o}v\'{a}ri, V.T. S\'{o}s, and P. Tur\'an, On a problem of K. Zarankiewicz,
\emph{Colloq. Math.}, {\bf 3} (1954), 50--57.

\bibitem{MN}
J. Ma and B. Ning, Stability results on the circumference of a graph, accepted by
Combinatorica on Jan. 2019, see also arXiv:1708.00704v2.

\bibitem{MM63}
J. Moon and L. Moser, On hamiltonian bipartite graphs, \emph{Israel J. Math.}, {\bf 1} (1963), 163--165.

\bibitem{NV05}
A. Naor and J. Verstra\"{e}te, A note on bipartite graphs without $2k$-cycles,
\emph{Combin. Probab. Comput.}, {\bf 14} (2005), 845--849.

\bibitem{S95}
C.N. S\'{a}rk\"{o}zy, Cycles in bipartite graphs and an application in number theory,
\emph{J. Graph Theory}, {\bf 19}(1995) 323--331.

\bibitem{W76}
D.R. Woodall, Maximal circuits of graphs I,
\emph{Acta Math. Acad. Sci. Hungar.}, {\bf 28} (1976), 77--80.
\end{thebibliography}
\end{document}